\documentclass[11pt,reqno]{amsart}
\usepackage{amscd,graphicx,psfrag}

\newcommand{\p}{\partial}
\newcommand{\Z}{\mathbb Z}

\newcommand{\M}{\mathcal M}

\renewcommand{\S}{\mathcal S}

\renewcommand{\phi}{\varphi}

\newcommand{\ind}{\operatorname{ind}}

\newcommand{\Tr}{\operatorname{Tr}}

\renewcommand{\Re}{\operatorname{Re}}

\newcommand{\sign}{\operatorname{sign}}
\newcommand{\W}{\operatorname{W}}
\newcommand{\sw}{\mathbf{sw}}
\newcommand{\Sign}{\operatorname{Sign}}
\newcommand{\Dir}{\operatorname{Dir}}
\newcommand{\SW}{\operatorname{SW}}
\newcommand{\FO}{\operatorname{FO}}

\newtheorem{theorem}{Theorem}[section]

\newtheorem{proposition}[theorem]{Proposition}
\newtheorem{corollary}[theorem]{Corollary}
\newtheorem*{conjecture}{Conjecture}

\theoremstyle{definition}

\def\acknowledgementname{Acknowledgements.}
\newenvironment{acknowledgement}
   {
    \vspace*{\baselineskip}
    \noindent{\small\textbf{\acknowledgementname}}
    \unskip\noindent}{}

\begin{document}

\title{Casson--type invariants from the Seiberg--Witten equations}%
\thanks{The first author was partially supported by NSF Grants
1105234 and 1065827. The second author was partially supported by NSF
Grant 1065905.}
\author[Daniel Ruberman]{Daniel Ruberman}
\address{Department of Mathematics, Brandeis University, Waltham, MA 02454}
\email{\rm{ruberman@brandeis.edu}}
\author[Nikolai Saveliev]{Nikolai Saveliev}
\address{Department of Mathematics, University of Miami, Coral Gables, FL 33124}
\email{\rm{saveliev@math.miami.edu}}

%\keywords{Casson invariant, Rohlin invariant, Floer homology, flat
%moduli spaces}

\subjclass[2000]{57M27, 57R57, 58J20, 58J28}

\begin{abstract}
This is a survey of our recent work~\cite{MRS1,MRS2,ruberman-saveliev:ded} with Tom Mrowka on Seiberg--Witten gauge theory and index theory for manifolds with periodic ends.  We explain how this work leads to a new invariant, which is related to the classical Rohlin invariant of homology 3-spheres and to the Furuta-Ohta invariant originating in Yang-Mills gauge theory.  We give some new calculations of our invariant for 4-dimensional mapping tori.
\end{abstract}

\maketitle%

\section{Introduction}

Since its inception in the mid-nineties, Seiberg--Witten gauge theory has had numerous applications in topology. The natural domain of this theory comprises simply connected closed oriented smooth 4-manifolds with $b_+ > 1$, where a straightforward count of irreducible solutions to the Seiberg--Witten equations produces a diffeomorphism invariant of the manifold. Here, $b_+$ stands for the number of positive eigenvalues of the intersection form on the second homology of the manifold. The theory has also been extended to manifolds with $b_+ = 1$ using wall-crossing formulas. The project surveyed in this article deals with the Seiberg--Witten theory for a class of manifolds having $b_+ = 0$, including manifolds with integral homology of $S^1 \times S^3$. In this case, the usual count of irreducible solutions to the Seiberg--Witten equations  depends heavily on metric and perturbation but in a joint project with Tom Mrowka, we succeeded in \cite{MRS1} in defining a diffeomorphism invariant by countering this dependence by a correction term. 

The correction term is of great interest in its own right. Its definition was inspired by the work of Weimin Chen \cite{chen:casson} and Yuhan Lim \cite{lim:b1}, who counted irreducible solutions to the Seiberg--Witten equations on a $3$-dimensional homology sphere. This count is not well-defined because of potential contributions from the reducible solutions to the equations as one varies the metric and perturbation in a one-parameter family; this issue is analogous to the one we encounter in the $4$-dimensional case when $b_+= 0$. To get a diffeomorphism invariant, Chen and Lim (independently,  following a suggestion of Kronheimer) added a correction term which is a combination of $\eta$--invariants of Atiyah--Patodi--Singer~\cite{aps:I}. Equivalently, their correction term can be expressed as a combination of the signature of a compact spin $4$-manifold with boundary the homology sphere, and the index of the spin Dirac operator on said $4$-manifold with the Atiyah--Patodi--Singer boundary conditions.

Our correction term is similar to that of Chen and Lim but requires a considerably more complicated analytical setting because our definition involves the spin Dirac operator on a non-compact manifold with a periodic end.  A fundamental analytical issue is therefore ensuring that this Dirac operator is Fredholm and hence has a well-defined index.  In dealing with this issue, we relied on the study of the Fredholm properties of elliptic operators on manifolds with periodic ends initiated by Taubes \cite{T}. We further developed this theory, which allowed us to prove the well-definedness of our invariant in \cite{MRS1}. It also led us in \cite{MRS2} to a general index theorem for end-periodic Dirac operators in the spirit of the Atiyah--Patodi--Singer theorem \cite{aps:I}, complete with a new $\eta$-invariant. A special case of this theorem is described in Section \ref{S:formula}.

Our interest in Seiberg--Witten invariants of manifolds with the homology of $S^1 \times S^3$ is mainly explained by the prominent role these manifolds play in low-dimensional topology\,: several outstanding problems concerning homology cobordisms and the Rohlin invariant can be translated using the doubling construction to problems about a homology $S^1 \times S^3$.  Moreover, the classification of smooth manifolds homotopy equivalent to $S^1 \times S^3$ is a basic problem for the theory of non-simply-connected $4$-manifolds.  We refer the reader to our paper \cite{ruberman-saveliev:survey} and to Section \ref{S:inv} of current paper for details.

It should be pointed out that in \cite{ruberman-saveliev:survey} we studied another set of gauge theoretic invariants of a homology $S^1 \times S^3$ arising from Donaldson gauge theory. As in Seiberg--Witten theory, the study of solutions to the Yang-Mills equations on a manifold with $b_+ = 0$ has some subtleties that are not present in the usual theory of Donaldson invariants. Among the invariants studied in \cite{ruberman-saveliev:survey} is one due to Furuta and Ohta~\cite{furuta-ohta} of manifolds with the $\Z[\Z]$ homology of $S^1 \times S^3$.  We conjecture that this Furuta--Ohta invariant matches the Seiberg--Witten invariants described in this paper. Our conjecture can be viewed as an extension of the Witten conjecture~\cite{witten} comparing Donaldson and Seiberg--Witten invariants to manifolds with $b_+ = 0$. It is straightforward to verify that,  for manifolds of the type $S^1 \times \Sigma$, where $\Sigma$ is an integral homology sphere, the Furuta-Ohta invariant reproduces the Casson invariant of $\Sigma$.  Moreover, our invariant in this product case is equal to the Seiberg--Witten invariant of Chen and Lim.  Thus the conjecture in the product case follows from these observations together with the theorem proved by Lim~\cite{lim:swcasson} that Chen and Lim's invariant is the Casson invariant.  In this paper, we verify the conjecture for the more general case of mapping tori of orientation preserving finite order diffeomorphisms $\tau: \Sigma \to \Sigma$ without fixed points; this is the only original result of this paper, and we provide its complete proof. While we have not been able to handle the situation when $\tau$ has fixed points in full generality, we verified the conjecture in some special cases in \cite{ruberman-saveliev:ded}.

\acknowledgement{\quad This paper grew out of a joint project with Tom Mrowka;
we truly appreciate his ongoing collaboration.  We are also thankful to
Liviu Nicolaescu and Weimin Chen for sharing their insight on the material
discussed in the last two sections.}

%%%%%%%%%%%%%%%%%%%%%%%%%%%%%%%%%%%%%%%%%%%%%%%%%%%%%%%%%%%%%%%%%%%%%%%%%%%%%%

\section{Seiberg--Witten moduli spaces}

A homology $S^1 \times S^3$ is a smooth closed oriented 4-manifold $X$ such that $H_* (X;\mathbb Z) = H_* (S^1 \times S^3;\mathbb Z)$. One way to obtain such a manifold is to furl up a smooth homology cobordism $W$ from an integral homology 3-sphere $\Sigma$ to itself by gluing the two boundary components of $W$ together via the identity map. If $W$ is the product cobordism, this construction will result in $X = S^1 \times \Sigma$, and if $W$ is the mapping cylinder of $\tau: \Sigma \to \Sigma$ the manifold $X$ will be the mapping torus of $\tau$.

The Seiberg--Witten invariant of $X$ that we wish to define will depend on a choice of generator $1 \in H^1 (X;\mathbb Z) = \mathbb Z$, called a \emph{homology orientation}. The invariant will prove to be independent of several other choices, which are however necessary to just write the Seiberg--Witten equations. These are the choices of spin structure, Riemannian metric $g$, and perturbation $\beta \in \Omega^1 (X, i\mathbb R)$. The manifold $X$ has two different spin structures, corresponding to the fact that $H^1 (X;\Z/2) = \Z/2$. Since these spin structures are the same when viewed as $\rm{spin}^c$ structures, our invariant will be independent of this choice. The independence of $g$ and $\beta$ is much less obvious, and proving it is a major part of this project. 

The Seiberg--Witten equations \cite{KM} are a system of non-linear partial differential equations on triples $(A,s,\phi)$, where $A$ is a $U(1)$ connection on the determinant bundle of the spin bundle, $\phi$ is a positive chiral spinor of $L^2$ norm one, and $s \ge 0$ is a real number. The equations read 
\[
F^+_A - s^2\cdot\tau(\varphi) = d^+\beta, \quad D^+_A (X,g) (\varphi) = 0,
\]

\noindent
where $F_A^+ \in \Omega^2_+ (X;i\mathbb R)$ is the anti-self-dual part of the curvature, $D^+_A (X,g)$ is the chiral Dirac operator on $X$, and $\tau(\phi)$ is a quadratic form in $\phi$ whose exact nature is immaterial for this paper. The gauge group, which consists of the maps $u: X \to S^1$, acts on the set of solutions of this system by the rule $(A,s,\varphi) \to (A-u^{-1} du,s,u\cdot\phi)$. The gauge equivalence classes of solutions $(A,s,\phi)$ form the Seiberg--Witten moduli space $\M (X,g,\beta)$. Solutions are called \emph{reducible} if $s = 0$, and \emph{irreducible} otherwise.

Note that these are the blown up equations of Kronheimer and Mrowka \cite{KM}. For $s > 0$, the map $(A,s,\varphi) \to (A, s\cdot\varphi)$ would take us back to the original Seiberg--Witten equations, but the reducibles now appear as the boundary points of the moduli space rather than as singularities. This apparently modest change in perspective turns out to be crucial for the analysis of the change in the moduli space in a path of perturbations and metrics that arises in the proof of Theorem~\ref{lsw}.  Also note that for $X$ a homology $S^1 \times S^3$, any $\omega \in \Omega^2_+ (X,i\mathbb R)$ normally used as a perturbation is of the form $\omega = d^+\beta$ because $H^2_+ (X;\mathbb Z) = 0$.

\begin{theorem}\label{T:reg}
Let $g$ be a metric on $X$. For a generic $\beta$, the moduli space $\M(X,g,\beta)$ is a compact oriented 0-dimensional manifold with no reducibles. 
\end{theorem}

For a proof, see \cite[Proposition 2.2\,]{MRS1}. Any pair $(g,\beta)$ as in Theorem \ref{T:reg} will be called \emph{regular}. Given a regular pair, denote by $\#\M(X,g,\beta)$ the signed count of points in the moduli space $\M(X,g,\beta)$. In general, this count will depend of the choice of $(g,\beta)$. 

To quantify this dependence, take two regular pairs $(g_0,\beta_0)$ and $(g_1,\beta_1)$ and connect them by a path $(g_t,\beta_t)$. This path can be chosen so that it goes through at most finitely many non-regular pairs $(g_t,\beta_t)$. The moduli spaces $\M(X,g_t,\beta_t)$ at such pairs will have reducibles which will prompt jumps in the count, see Figure \ref{fig:param}.

\medskip

\begin{figure}[!ht]
\centering
\psfrag{s}{$s$}
\psfrag{t}{$t$}
\psfrag{t=0}{$t=0$}
\psfrag{t=1}{$t=1$}
\includegraphics{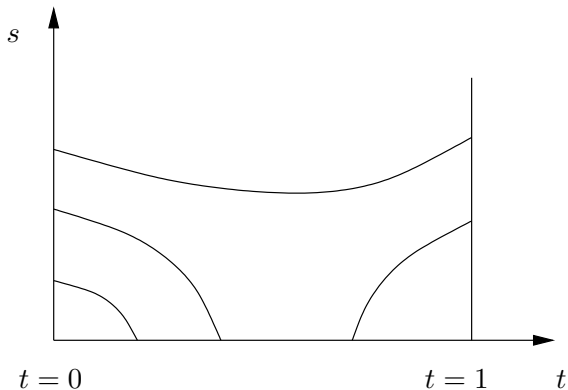}
\caption{The parameterized moduli space}
\label{fig:param}
\end{figure}

%%%%%%%%%%%%%%%%%%%%%%%%%%%%%%%%%%%%%%%%%%%%%%%%%%%%%%%%%%%%%%%%%%%%%%%%%%%%

\section{The correction term}
Let $X$ be a homology $S^1 \times S^3$ and $p: \tilde X \to X$ the infinite cyclic cover corresponding to the generator $1 \in H^1 (X;\mathbb Z)$ provided by the choice of homology orientation. Given a submanifold $Y \subset X$ dual to this generator, cut $X$ open along $Y$ to obtain a cobordism $W$ from $Y$ to itself and write
\[
\tilde X\;=\;\bigcup_{n \in \mathbb Z}\; W_n\quad\text{with}\quad W_n = W.
\]
Define
\[
\tilde X_+\;=\;\bigcup_{n \ge 0}\; W_n\quad\text{and}\quad
Z_+\;=\;Z \cup \tilde X_+
\]

\smallskip\noindent
for any smooth compact spin 4-manifold $Z$ with boundary $Y$. The manifold $Z_+$ is a manifold with periodic end in the sense of Taubes \cite{T}. Our choices of metric $g$, perturbation $\beta$ and spin structure on $X$ lift to $\tilde X_+$ and extend to a metric $g$, perturbation $\beta$ and spin structure on $Z_+$. With respect to Sobolev $L^2$--completions, the spin Dirac operator $D^+ (Z_+,g): L^2_1 (Z_+,S^+) \to L^2 (Z_+,S^-)$ is a bounded operator on the non-compact manifold $Z_+$. The following is proved in \cite[Theorem 3.1]{MRS1}.

\begin{theorem}\label{T:fred}
The perturbed Dirac operator $D^+ (Z_+,g) + \beta$ is Fredholm for any regular pair $(g,\beta)$.
\end{theorem}

Note that the operator $D^+ (Z_+,g) + \beta$ is complex linear and, for any regular pair $(g,\beta)$, define the \emph{correction term}
\[
w(X,g,\beta)\;=\;\ind_{\,\mathbb C} (D^+ (Z_+,g) + \beta)\, +\, \frac 1 8\,\sign Z.
\]

\begin{theorem}
The correction term $w(X,g,\beta)$ is independent of the choices of manifolds $Z$ and $Y \subset X$, and the way $g$, $\beta$ and the spin structure are extended to $Z$.
\end{theorem}

For a proof, see \cite[Proposition 3.2\,]{MRS1}. Since the manifold $Z_+$ is not compact, the index $\ind_{\,\mathbb C}(D^+ (Z_+,g)+\beta)$ is sensitive to changes in metric $g$ and perturbation $\beta$. This makes the correction term $w(X,g,\beta)$ depend on the choice of regular pair $(g,\beta)$. Like with the count $\#\,\M(X,g,\beta)$ in the previous section, quantifying this dependence involves tracing the jumps in $\ind_{\,\mathbb C} (D^+ (Z_+,g_t) + \beta_t)$ along a generic path $(g_t,\beta_t)$ between two regular pairs. This requires a good understanding of the index theory on manifolds with periodic ends. This theory was initiated by Taubes \cite{T}. In our paper \cite{MRS1} we developed it far enough to calculate the jumps in the index and to match them with the jumps in the count $\#\,\M(X,g,\beta)$. This has led to the invariant described in the next section.

%%%%%%%%%%%%%%%%%%%%%%%%%%%%%%%%%%%%%%%%%%%%%%%%%%%%%%%%%%%%%%%%%%%%%%%%%%%%%%

\section{The invariant}\label{S:inv}
Let $X$ be a smooth oriented 4-manifold with the integral homology of $S^1 \times S^3$ and with a fixed homology orientation $1 \in H^1 (X;\Z)$. Given a regular pair $(g,\beta)$, define
\[
\lambda_{\,\SW} (X) = \#\,\M(X,g,\beta) - w(X,g,\beta).
\]

\begin{theorem}\label{lsw}
The invariant $\lambda_{\,\SW} (X)$ is well defined, that is, independent of the choice of regular pair $(g,\beta)$. Moreover, the reduction of $\lambda_{\,\SW} (X)$ modulo 2 equals the Rohlin invariant of $X$.
\end{theorem}

This theorem is proved in \cite{MRS1}. Recall that the Rohlin invariant of $X$ is defined as $\sign Z/8\pmod 2$, where $Z$ is a smooth compact spin manifold with boundary $Y \subset X$ dual to the generator $1 \in H^1 (X;\Z)$. The proof that $\lambda_{\,\SW} (X)$ reduces to the Rohlin invariant modulo 2 requires a stronger version of Theorem \ref{T:fred} stating that the \emph{unperturbed} Dirac operator $D^+(Z_+,g)$ is Fredholm for a generic metric $g$ on $X$. This was proved in our paper \cite{ruberman-saveliev:ggt}. 

The fact that $\lambda_{\,\SW} (X)$ reduces modulo 2 to the Rohlin invariant opens the door to potential applications of $\lambda_{\,\SW} (X)$ to several old problems concerning homology cobordisms. These problems are described in \cite{ruberman-saveliev:survey}, where we attempted to address them using a different gauge theoretic invariant $\lambda_{\,\FO} (X)$ called the \emph{Furuta--Ohta invariant}. The latter is defined using Donaldson theory as roughly one quarter times a count of irreducible instantons in the trivial $SU(2)$ bundle on $X$. Here, $X$ must be a $\mathbb Z[\mathbb Z]$ homology $S^1 \times S^3$ meaning that, in addition to its having the integral homology of $S^1 \times S^3$, its infinite cyclic cover has the integral homology of $S^3$. This additional condition is satisfied, for instance, when a generator of $H_3 (X;\mathbb Z)$ is carried by an integral homology sphere $Y \subset X$. Later in the paper, we will present evidence for the following conjecture.

\begin{conjecture}
Let $X$ be a $\mathbb Z[\mathbb Z]$ homology $S^1 \times S^3$ with a fixed orientation and homology orientation. Then 
\begin{equation}\label{E:conj}
\lambda_{\,\FO} (X) = - \lambda_{\,\SW} (X).
\end{equation}
\end{conjecture}

Let us briefly explain how this conjecture is relevant to the study of manifolds with the homotopy type of $S^1\times S^3$. High-dimensional surgery theory predicts a calculation of the structure set  $\S_{\rm{Diff}}(S^1\times S^3)$, consisting of homotopy equivalences 
$$
f:X \to S^1\times S^3
$$
with $X$ a smooth manifold. The surgery exact sequence (\cite{wall:book}; see~\cite{kirby-taylor:surgery} for the calculations cited below) would predict that the cokernel of the map 
$$
N(S^1\times S^3 \times I, S^1\times S^3 \times \partial I) \overset{\sigma}{\longrightarrow} L_5(\Z[\Z])
$$
between the normal maps on $S^1\times S^3 \times I$ and the surgery group $L_5(\Z[\Z])$ acts freely on $\S(S^1\times S^3)$.  Both of those groups are isomorphic to $\Z$, where the isomorphism is given by the signature of a codimension-one submanifold dual to a generator of $H^1(S^1\times S^3)$.  A computation involving Rohlin's theorem implies that the map $\sigma$ is actually multiplication by $2$; the upshot is that one might expect a smooth manifold homotopy equivalent to $S^1\times S^3$, detected by the Rohlin invariant.  

The conjecture above would imply that such a manifold does not exist. For it is automatic from the definition that $\lambda_{\,\FO} (X) =0$ for any $X$ homotopy equivalent to $S^1 \times S^3$ and so the conjecture would imply that $ \lambda_{\,\SW} (X)$ vanishes as well.  But then the last part of Theorem~\ref{lsw} would give the vanishing of the Rohlin invariant.  We remark that although there are exotic homotopy equivalences produced by the action of $L_5(\Z[\pi])$ on the structure set~\cite{scharlemann:phs,akbulut:scharlemann} there does not seem to be any known example of an exotic smooth structure on a $4$-manifold produced by this action; the case of $S^1\times S^3$ is the most basic test case.

To conclude this section, we will mention that the invariant $\lambda_{\,\SW} (X)$ was extended in \cite{MRS1} to a wider class of negative-definite 4-manifolds $X$, those with $H_1 (X;\mathbb Z) = \mathbb Z$ but not necessarily vanishing $H_2 (X;\mathbb Z)$. Such manifolds are encountered, for example, in the study of non-K{\"a}hler complex surfaces, see \cite{okonek-teleman:b+0}.

%%%%%%%%%%%%%%%%%%%%%%%%%%%%%%%%%%%%%%%%%%%%%%%%%%%%%%%%%%%%%%%%%%%%%%%%%%%%%%

\section{A formula for $\lambda_{\,\SW}$}\label{S:formula}
In this section we express the invariant $\lambda_{\,\SW}(X)$ solely in terms of the manifold $X$, without referring to auxiliary end-periodic manifolds. This formula will follow from the general index theorem for end-periodic operators proved in \cite{MRS2}. 

Before we state the theorem we need a few definitions. Let $X$ be a homology $S^1 \times S^3$ with a fixed orientation and a fixed homology orientation and choose a smooth function $f: X \to S^1$ so that $[df] = 1 \in H^1 (X;\mathbb Z)$. For any choice of metric $g$, consider the holomorphic family
\[
D^{\pm}_z = D^{\pm} (X,g) - \ln z\cdot df,\quad z \in \mathbb C^*, 
\]
of twisted Dirac operators on $X$. All of these operators have index zero. It follows from \cite{ruberman-saveliev:ggt} that, for a generic metric $g$, the operators $D^{\pm}_z$ are invertible away from a discrete set $\S \subset \mathbb C^*$ and moreover, the set $\S$ can be chosen to be disjoint from the unit circle $|z| = 1$. In particular, all of the operators $D^{\pm}_z$ with $|z| = 1$ are invertible. The set $\S$ is called the \emph{spectral set}; one can show that it is independent of the choice of $f$. The \emph{$\eta$-invariant} is defined in \cite{MRS2} by the formula
\[
\eta (X,g)\;=\;\frac 1 {\pi i}\,\int_0^{\infty} \oint_{|z| = 1} \Tr\,\left(df\cdot D^+_z e^{-t D^-_z D^+_z}\right)\,\frac {dz} {z}\;dt.
\]

\smallskip
To get a better grip on $\eta (X,g)$ consider the special case of $X = S^1 \times Y$ with a product metric and spin structure so that $D^+ (X,g) = \p/\p t - D$, where $D$ is the self-adjoint Dirac operator on $Y$. Choose $f: X \to S^1$ to be the projection onto the circle factor then the above formula will simplify to 

\[
\eta (S^1 \times Y,g)\;=\;\frac 1 {\sqrt{\,\pi}}\,\int_0^{\infty} t^{-1/2}\,\Tr\,(D\,e^{-t D^2})\,dt.
\]

\medskip\noindent
The right hand side of this formula matches the $\eta$-invariant $\eta_{\,\Dir}(Y)$ of Atiyah, Patodi and Singer \cite{aps:I} hence we conclude that $\eta (S^1 \times Y,g) = \eta_{\,\Dir}(Y)$. According to \cite{aps:I}, the $\eta$-invariant $\eta_{\,\Dir} (Y)$ also equals the value at $s = 0$ of the meromorphic extension of the function  
\[
\sum_{\lambda \ne 0}\; \sign \lambda\,|\lambda|^{-s}
\]
defined for sufficiently large $\Re(s)$ by summing over the spectrum of $D$. Thus one can say that $\eta_{\,\Dir} (Y)$ measures the asymmetry of the spectrum of $D$. Similarly, we show in \cite{MRS2} that $\eta (X,g)$ measures the asymmetry of the spectral set $\S$ with respect to the unit circle: the integral defining $\eta (X,g)$ can be viewed as a regularization of the difference between the number of spectral points outside of the circle $|z| = 1$ and the number of those inside. 

The following is a special case of the end-periodic index theorem proved in \cite{MRS2} for Dirac-type operators in all dimensions divisible by four.

\begin{theorem}\label{T:index}
Let $X$ be a homology $S^1 \times S^3$ and $Z_+$ an end-periodic manifold whose end is modeled on the infinite cyclic cover of $X$. For a generic metric $g$ on $X$ making $D^+ (Z_+,g)$ Fredholm, we have 
\[
\ind_{\,\mathbb C} D^+ (Z_+,g)\;=\;\int_Z \widehat A\;-\,\int_Y \omega\;+\,\int_X df\wedge \omega\,-\,\frac 1 2 \,\eta (X,g),
\]
where $Y \subset X$ is a submanifold dual to $1 \in H^1 (X;\mathbb Z)$ and $\omega$ is a transgressed $\widehat A$--class given by $d\omega = \widehat A$. 
\end{theorem}

Under the assumption that $Y \subset X$ is chosen to have a normal neighborhood $N(Y) \subset X$ with product metric and $\operatorname{supp} f \subset N(Y)$, the index formula of this theorem can be simplified to
\begin{equation}\label{E:index}
\ind_{\,\mathbb C} D^+ (Z_+,g)\;=\;\int_Z \widehat A\,-\,\frac 1 2 \,\eta (X,g).
\end{equation}
In the special case of $X = S^1 \times Y$, this formula reduces to that of Atiyah, Patodi and Singer \cite{aps:I} for manifolds with product ends. 

\begin{corollary}
Let $X$ be a homology $S^1 \times S^3$ with a metric $g$ such that the pair $(g,\beta)$ with $\beta = 0$ is regular, and suppose that $Y \subset X$ can be chosen to have a normal neighborhood $N(Y) \subset X$ with product metric and $\operatorname{supp} f \subset N(Y)$. Then 
\[
\lambda_{\,\SW} (X)\;=\;\#\,\M(X,g,\beta)\,+\,\frac 1 8\,\eta_{\,\Sign} (Y)\,+\, \frac 1 2\,\eta (X,g),
\]
where $\eta_{\,\Sign} (Y)$ is the Atiyah--Patodi--Singer $\eta$-invariant of the odd signature operator on $Y$.
\end{corollary}

\begin{proof}
Using formula \eqref{E:index} together with the signature theorem of Atiyah, Patodi and Singer \cite{aps:I},
\[
\sign Z\;=\; \int_Z L\,-\,\eta_{\,\Sign} (Y),
\]
and keeping in mind that $\widehat A = - p_1/24$ and $L = p_1/3$ in degree four, we obtain
\begin{multline}\notag
w(X,g,\beta) 
\,=\, \ind_{\,\mathbb C} D^+ (Z_+,g) + \frac 1 8\,\sign Z \\
= -\frac 1 {24}\,\int_Z p_1  - \frac 1 2\,\eta (X,g) + \frac 1 {24}\,\int_Z p_1
- \frac 1 8\,\eta_{\,\Sign} (Y),
\end{multline}

\medskip\noindent
and the statement obviously follows. 
\end{proof}

Finally, we mention that the requirement that $(g,\beta)$ be a regular pair for $\beta = 0$ is not essential. For an arbitrary regular pair $(g,\beta)$, Theorem \ref{T:index} will hold for the perturbed Dirac operator $D^+ (Z_+,g) + \beta$ once the family $D^{\pm}_z$ used to define $\eta (X,g)$ is replaced with the perturbed family $D^{\pm}_z + \beta$. 

%%%%%%%%%%%%%%%%%%%%%%%%%%%%%%%%%%%%%%%%%%%%%%%%%%%%%%%%%%%%%%%%%%%%%%%%%%%%%%

\section{The product case} 
Let $\Sigma$ be an oriented integral homology sphere and $X = S^1 \times \Sigma$. We will work with product metrics $g$ on $X$ and with perturbations $\beta$ which are constant in the direction of $S^1$, and from now on we will suppress both in our notations. It is a well-known fact that $\M(S^1 \times \Sigma)$ equals $\M (\Sigma)$, the Seiberg--Witten moduli space in dimension 3, see for instance \cite{KM}.  
Together with the above discussion of the $\eta$-invariants, this implies that
\[
\lambda_{\,\SW} (S^1 \times Y)\;=\;\#\,\M(\Sigma)\,+\, \frac 1 8\;\eta_{\,\Sign} (Y)\,+\, \frac 1 2\;\eta_{\,\Dir} (Y).
\]
The right hand side of this equality was studied by Weimin Chen \cite{chen:casson} and Yuhan Lim \cite{lim:swcasson}. Lim showed that it equals, up to an overall sign, the Casson invariant $\lambda (\Sigma)$. Recall that $\lambda(\Sigma)$ is defined as one half times a signed count of the conjugacy classes of irreducible $SU(2)$ representations of $\pi_1 (\Sigma)$, see \cite{AMc}. A quick calculation with the Poincar\'e homology sphere $\Sigma(2,3,5)$ fixes the overall sign to be negative, hence we conclude that
\[
\lambda_{\,\SW} (S^1 \times \Sigma)\;=\;-\lambda (\Sigma).
\]
This confirms our conjecture \eqref{E:conj} in the product case since we showed in \cite{ruberman-saveliev:mappingtori} that $\lambda_{\,\FO} (S^1 \times \Sigma) = \lambda (\Sigma)$. We wish to extend the above calculation of $\lambda_{\,\SW}$ and verify the conjecture in the next simplest case, that of mapping tori with finite order monodromy.

%%%%%%%%%%%%%%%%%%%%%%%%%%%%%%%%%%%%%%%%%%%%%%%%%%%%%%%%%%%%%%%%%%%%%%%%%%%%%%

\section{Mapping tori: free case}
Let $\Sigma$ be an oriented integral homology sphere and $X$ the mapping torus of an orientation preserving diffeomorphism $\tau: \Sigma \to \Sigma$ of order $n$. Suppose that $\tau: \Sigma \to \Sigma$ has no fixed points then the quotient $\Sigma' = \Sigma/\tau$ is a homology lens space, and $X$ can be viewed as the total space of the circle bundle $\pi: X \to \Sigma'$ whose Euler class generates $H^2(\Sigma';\Z) = \Z/n$. Let $i\eta$ be the connection form of this bundle and $g'$ a metric on $\Sigma'$. Endow $X$ with the metric $g = \eta^2 + \pi^*g'$. Furthermore, given a perturbation 1-form $\beta'$ on $\Sigma'$, lift it to the perturbation 1-form $\beta = \pi^* \beta'$ on $X$. 

Note that the unique $\rm{spin}^c$-structure on $X$ is pulled back from $\Sigma'$. Therefore, according to \cite[Theorem B]{baldridge}, the pull-back map 
\[
\pi^*:\;\M^* (\Sigma',g',\beta') \to \M^*(X,g,\beta)
\]
provides a bijective correspondence between the \emph{irreducible} portions of the Seiberg--Witten moduli spaces on $\Sigma'$ and on $X$. Moreover, for a generic choice of $(g',\beta')$, there are no reducibles on $\Sigma'$ and the above correspondence is an orientation preserving diffeomorphism between compact oriented 0-dimensional manifolds. The full moduli space $\M (X,g,\beta)$ may in principle contain reducibles because perturbation forms $\beta = \pi^*\beta'$ as above are not dense in the space of all perturbations. That $\M (X,g,\beta)$ is actually free of reducibles can be verified by a Fourier analysis calculation using the observation that the infinite cyclic cover of $X$ is isometric to a product. 

The same observation tells us that $\eta (X,g) = \eta_{\,\Dir} (\Sigma)$ hence we conclude that
\[
\lambda_{\,\SW} (X)\,=\,\sum_{\sigma'}\;\#\,\M (\Sigma',\sigma')\,+\,
\frac 1 2\,\eta_{\,\Dir}\,(\Sigma)\,+\,\frac 1 8\,\eta_{\,\Sign} (\Sigma),
\]
where we broke the moduli space $\M(\Sigma',g',\beta')$ into a disjoint union of the moduli spaces $\M(\Sigma',\sigma')$ corresponding to the $n$ distinct $\rm{spin}^c$-structures on $\Sigma'$, and suppressed metrics and perturbations in our notations. On the other hand, consider the rational number
\[
\sw^0 (\Sigma',\sigma')\,=\,\#\,\M(\Sigma',\sigma')\,+\,\frac 1 2\,
\eta_{\,\Dir}\;(\Sigma',\sigma')\,+\,\frac 1 8\;\eta_{\,\Sign} (\Sigma'),
\]
where $\eta_{\,\Dir}\,(\Sigma',\sigma')$ stands for the $\eta$-invariant of the $\rm{spin}^c$ Dirac operator corresponding to the ${\rm spin}^c$ structure $\sigma'$. Lim \cite{lim:b1} showed that $\sw^0 (\Sigma',\sigma')$ is a topological invariant, and Marcolli and Wang \cite[Theorem 1.1]{marcolli-wang} later proved that\footnote{The orientation conventions in Marcolli--Wang \cite{marcolli-wang} differ from ours, which accounts for the extra negative sign in our formula compared to theirs, cf. Nicolaescu \cite{nic-sw}.}
\begin{equation}\label{E:walker}
\sum_{\sigma'}\;\sw^0 (\Sigma',\sigma')\,= -\lambda_{\W} (\Sigma'),
\end{equation}
where $\lambda_{\W} (\Sigma')$ is the Casson--Walker invariant normalized as in Lescop \cite{lescop}, meaning that $\lambda_{\W}(\Sigma')$ equals $n/2$ times the Casson--Walker invariant defined in Walker \cite{walker}. Combining the last three formulas, we obtain
\begin{multline}\notag
\lambda_{\,\SW} (X)\,=\,-\lambda_{\W}(\Sigma')\, + \frac 1 8\,
\left(\eta_{\,\Sign}(\Sigma) - n\cdot\eta_{\,\Sign} (\Sigma')\right) \\ 
+ \frac 1 2\,\left(\eta_{\,\Dir}(\Sigma) - \sum_{\sigma'}\;\eta_{\,\Dir} 
(\Sigma',\sigma')\right).
\end{multline}

We will next identify the last two terms on the right. The last term actually vanishes: since $\eta_{\,\Dir}(\Sigma',\sigma')$ are just the $\eta$-invariants of the spin Dirac operator twisted by representations $\alpha: \pi_1 \Sigma' \to U(1)$, their sum over all $\alpha$ clearly equals the $\eta$-invariant of the Dirac operator on $\Sigma$. Using the $\rho$--invariants of \cite{aps:II} we can write
\[
\eta_{\,\Sign}(\Sigma) - n\cdot\eta_{\,\Sign} (\Sigma')\,=\, 
\sum_{\alpha}\;\rho_{\alpha} (\Sigma'),
\]
where the summation extends to all the representations $\alpha: \pi_1 
\Sigma' \to U(1)$. To calculate $\rho_{\alpha} (\Sigma')$, view $\Sigma'$ 
as the result of $(n/q)$--surgery along a knot $k$ in an integral homology 
sphere $Y$, where $q$ is relatively prime to $n$ and $0 < q < n$. 

\begin{proposition}
Let $\alpha: \pi_1 \Sigma' \to U(1)$ send the meridian of the knot $k$ to 
$e^{2\pi i m/n} \in U(1)$. Then 
\[
\rho_{\alpha} (\Sigma')\,=\,\rho_{\alpha} (L(n,q))\,-\,\sign^{m/n} k,
\]
where $L(n,q)$ is the lens space obtained by $(n/q)$-surgery on the unknot
in $S^3$, and $\sign^{m/n} k$ is the Tristram--Levine equivariant knot 
signature of $k$. 
\end{proposition}

\begin{proof}
For knots in $S^3$, this is essentially the formula of Kirk, Klassen, and Ruberman stated on page 388 of \cite{kkr} right before Theorem 4.4. We offer here a slight modification of their argument that extends that formula to knots in arbitrary homology spheres. Start as in the second paragraph on page 385 of \cite{kkr} with the manifold $U = ([0,1]\times Y)\,\cup\,H$, where $H$ is a 2-handle attached to $Y$ along the knot $k$ with zero framing. Choose a Seifert surface $F$ for $k$ and let $\bar F$ be the union of $F$ pushed slightly into $[0,1]\times Y$ and the core of the 2-handle. Remove a tubular neighborhood of $\bar F$ from $U$ to obtain a manifold $W$ with boundary. The boundary of $W$ will consist of three components\,: the 3-manifold obtained from $Y$ by 0-surgery on $k$, the product $\bar F \times S^1$, and the homology sphere $Y$. Let $B$ be a handlebody of genus equal to the genus of $F$, and $V$ a simply-connected smooth 4-manifold with boundary $Y$. Let
\[
M_X = V\,\cup\,W\,\cup\,(B\times S^1).
\]
By Novikov additivity, $\sign M_X = \sign V + \sign W$ and $\sign_{\alpha} M_X = \sign V + \sign_{\alpha} W$ since $\alpha$ restricts to a trivial representation on $\pi_1 V$. The calculation of $\sign W$ proceeds as before, and the contributions from $V$ cancel in the expression $3\,\sign N - \sign_{\alpha} N$ on page 388 of \cite{kkr}.
\end{proof}

Together with the surgery formula for the Casson--Walker invariant, see 
Walker \cite[Proposition 6.2]{walker}, this gives the following formula\,:
\begin{multline}\notag
\lambda_{\,\SW}(X)\,=\,-n\cdot\lambda(Y) - \frac 1 8\,\sum_{m=0}^{n-1}\;
\sign^{m/n} k - \frac q 2\,\Delta_k''(1) \\ - \lambda_{\W}(L(n,q)) + 
\frac 1 8\,\sum_{\alpha}\,\rho_{\alpha} (L(n,q)),
\end{multline}
where $\Delta_k (t)$ is the Alexander polynomial of the knot $k \subset Y$ normalized so that $\Delta_k (1) = 1$ and $\Delta_k (t^{-1}) = \Delta_k (t)$. This should be compared with the formula for $\lambda_{\,\FO} (X)$ from our paper \cite{ruberman-saveliev:mappingtori}\,:
\[
\lambda_{\FO}(X) = n\cdot\lambda(Y) + \frac 1 8\,\sum_{m=0}^{n-1}\; 
\sign^{m/n} k + \frac q 2\,\Delta_k''(1).
\]
We conclude that 
\[
\lambda_{\,\SW}(X) + \lambda_{\FO}(X)\;=\;- \lambda_{\W}(L(n,q)) +
\frac 1 8\,\sum_{\alpha}\,\rho_{\alpha} (L(n,q)),
\]
where
\[
\lambda_{\W}(L(n,q))\,=\,\frac 1 8\,\sum_{\alpha}\,\rho_{\alpha} (L(n,q))\,= -\frac 1 8\;\sum_{k=0}^{n-1}\;\cot\left(\frac {\pi qk}{n}\right)\cot\left(\frac {\pi k}{n}\right),
\]
see \cite[Proposition 6.3]{walker} for the Casson--Walker invariant and \cite[Proposition 2.12]{aps:II} for the $\rho$-invariants. This leads to the desired equality $\lambda_{\SW}(X) = -\lambda_{\FO}(X)$.

%%%%%%%%%%%%%%%%%%%%%%%%%%%%%%%%%%%%%%%%%%%%%%%%%%%%%%%%%%%%%%%%%%%%%%%%%%%%%

\section{Mapping tori: non-free case}
Assume now that $X$ is the mapping torus of a finite order diffeomorphism $\tau: \Sigma \to \Sigma$ of an oriented integral homology sphere $\Sigma$ which has fixed points. The quotient space $\Sigma' = \Sigma/\tau$ is then an integral homology sphere, and the projection $\Sigma \to \Sigma'$ is an $n$-fold branched cover with branch set a knot $k \subset \Sigma'$. We know from 
\cite{ruberman-saveliev:mappingtori} that 
\begin{equation}\label{E:one}
\lambda_{\FO} (X)\; =\; n\cdot \lambda(\Sigma') + \frac 1 8\,\sum_{m=0}^{n-1}\; 
\sign^{m/n} k,
\end{equation}

\smallskip\noindent
where $\lambda(\Sigma')$ is the Casson invariant of $\Sigma'$ and $\sign^{m/n} k$ are the Tristram--Levine equivariant knot signatures of $k$. As a first step towards computing $\lambda_{\SW} (X)$, we could use \cite{baldridge} to express $\M(X)$ in terms of the Seiberg--Witten moduli spaces on the \emph{orbifold} $\Sigma'$. While the Seiberg--Witten theory on orbifolds has been actively studied, see for instance Chen \cite{chen:orbifold}, one still lacks an orbifold version of the formula \eqref{E:walker} which was crucial for the calculation in the fixed point free case.

While the general case is still outstanding, we have been able in \cite{ruberman-saveliev:ded} to verify our conjecture in some special cases. To be specific, let $\Sigma = \Sigma(a_1,\ldots,a_n)$ be a Seifert fibered homology sphere oriented as a link of a Brieskorn--Hamm complete intersection singularity with real coefficients, and let $X$ be the mapping torus of the involution on $\Sigma$ induced by complex conjugation. We showed in \cite{MRS1} that, for a natural metric $g$ realizing the Thurston geometry on $\Sigma$, the pair $(g,0)$ is regular and $\M(X)$ is empty. Since the infinite cyclic cover of $X$ is isometric to a product, the invariant $\lambda_{\SW}(X)$ equals 
\[
\frac 1 2\,\eta_{\,\Dir} (\Sigma) + \frac 1 8\,\eta_{\,\Sign} (\Sigma),
\]
which in turn equals negative $\bar\mu (\Sigma)$, the $\bar\mu$--invariant of Neumann and Siebenmann, see \cite{ruberman-saveliev:ded}. The latter was identified with $\lambda_{\FO} (X)$  in \cite{ruberman-saveliev:mappingtori} thus leading to the conclusion that $\lambda_{\SW} (X) = -\lambda_{\FO}(X)$ in this case. 

\smallskip

%%%%%%%%%%%%%%%%%%%%%%%%%%%%%%%%%%%%%%%%%%%%%%%%%%%%%%%%%%%%%%%%%%%%%%%%%%%%%%

\end{document}